\documentclass[11pt]{amsart}
\usepackage[T1]{fontenc}
\usepackage{lmodern}
\usepackage{amsmath,amssymb,amsthm,mathrsfs,comment,colonequals}
\usepackage{tikz}
\usetikzlibrary{matrix,arrows,decorations.pathmorphing,cd,positioning,calc}
\usepackage{enumitem}
\usepackage{hyperref}
\hypersetup{colorlinks=true,linkcolor=blue,citecolor=blue,urlcolor=blue,
 pdftitle={Local-Cohen-Macaulay DG-rings with unequal finitistic projective dimensions},pdfauthor={Luigi Ferraro}}
\usepackage{cleveref}
\usepackage{microtype}
\numberwithin{equation}{section}
\newtheorem{theorem}{Theorem}[section]

\theoremstyle{definition}
\newtheorem*{ack}{Acknowledgements}

\theoremstyle{remark}

\crefname{lem}{lemma}{lemmas}
\Crefname{lem}{Lemma}{Lemmas}
\crefname{prop}{proposition}{propositions}
\Crefname{prop}{Proposition}{Propositions}
\crefname{proposition}{proposition}{propositions}
\Crefname{proposition}{Proposition}{Propositions}
\crefname{cor}{corollary}{corollaries}
\Crefname{cor}{Corollary}{Corollaries}
\crefname{corollary}{corollary}{corollaries}
\Crefname{corollary}{Corollary}{Corollaries}
\crefname{defn}{definition}{definitions}
\Crefname{defn}{Definition}{Definitions}
\crefname{definition}{definition}{definitions}
\Crefname{definition}{Definition}{Definitions}

\DeclareMathOperator{\Ext}{Ext}

\DeclareMathOperator{\pd}{\operatorname{projdim}}

\DeclareMathOperator{\RHom}{\mathbf{R}Hom}
\newcommand{\lotimes}{\otimes^{\mathbf L}}

\newcommand{\m}{\mathfrak m}

\DeclareMathOperator{\FPD}{FPD}
\DeclareMathOperator{\fpd}{fpd}
\DeclareMathOperator{\amp}{amp}
\newcommand{\D}{\mathsf D}
\newcommand{\RG}{\mathbf R\Gamma}


\begin{document}
\title[Finitistic dimensions of local-Cohen-Macaulay DG-rings]{Local-Cohen-Macaulay DG-rings with unequal finitistic projective dimensions}
\author[Ferraro]{Luigi Ferraro}
\address[Ferraro]{School of Mathematical and Statistical Sciences \\ University of Texas Rio Grande Valley \\ Edinburg, TX 78539, U.S.A}
\email[]{luigi.ferraro@utrgv.edu}
\date{}
\begin{abstract}
For every pair of integers $d,n\geq1$, we construct a commutative
noetherian local-Cohen-Macaulay DG-ring $A$ such that
$\dim H^0(A)=d$, $\amp(A)=n$, $\fpd(A)=d-n$, and $\FPD(A)=d$.
This gives a negative answer to the implication from
local-Cohen-Macaulayness to equality of the two finitistic projective
dimensions in a question of Bird, Shaul, Sridhar, and Williamson.
\end{abstract}
\maketitle

\section{Introduction}

The small and big finitistic projective dimensions of a ring
measure the projective dimensions of finitely generated and
arbitrary modules of finite projective dimension, respectively.
For a commutative noetherian local ring $R$, classic results
give
\[
\fpd(R)=\operatorname{depth}R
\qquad\text{and}\qquad
\FPD(R)=\dim R;
\]
see \cite[\S 1]{BSSW}. Consequently, these two invariants
coincide if and only if $R$ is Cohen-Macaulay.

Bird, Shaul, Sridhar, and Williamson studied the corresponding
invariants for commutative nonpositive noetherian DG-rings
with bounded cohomology. For such a DG-ring $A$, with
$d=\dim H^0(A)<\infty$ and $n=\amp(A)$, they proved that
\[
d-n\leq \FPD(A)\leq d.
\]
When $A$ is local, they also showed that $A$ is
local-Cohen-Macaulay if and only if $\fpd(A)=d-n$;
see \cite[Theorem~7.2 and Corollary~7.6]{BSSW}.
They asked \cite[Remark~7.7]{BSSW} whether the classic characterization extends
to this setting: does $\fpd(A)=\FPD(A)$ hold if and only
if $A$ is local-Cohen-Macaulay?

In this work, we show, by generalizing an example of Shaul \cite[Example 7.2]{ShaulCM}, that local-Cohen-Macaulayness does
not imply equality of the two finitistic projective dimensions.
For every field $k$ and every pair of integers $d,n\geq 1$,
we consider the trivial extension
\[
R=\frac{k[[x,y,z_1,\ldots,z_{d-1}]]}{(xy)},
\qquad
A=R\ltimes \bigl(R/(x)\bigr)[n].
\]
We prove that $A$ is local-Cohen-Macaulay and satisfies
\[
\dim H^0(A)=d,\qquad
\amp(A)=n,\qquad
\fpd(A)=d-n,\qquad
\FPD(A)=d.
\]
Thus the upper bound for the big finitistic projective
dimension can be attained even among local-Cohen-Macaulay
DG-rings, and the difference $\FPD(A)-\fpd(A)$ can be
arbitrarily large.

\begin{ack}
The author was partly supported by the Simons Foundation grant MPS-TSM-00007849.
\end{ack}

\section{Background}

In this section we collect the conventions and background used throughout the paper.

\subsection*{DG-rings and DG-modules}
A DG-ring is a graded ring $A=\bigoplus_{i\in\mathbb Z}A^i$
with a differential $\partial$ of degree $1$ satisfying
$\partial^2=0$ and
\[
\partial(ab)=\partial(a)b+(-1)^{|a|}a\partial(b)
\]
for homogeneous elements $a,b\in A$. We use commutativity in the
graded sense: $ab=(-1)^{|a||b|}ba$, and $a^2=0$ when $|a|$ is odd.
The DG-ring $A$ is nonpositive if $A^i=0$ for $i>0$.

A nonpositive DG-ring $A$ is noetherian if $H^0(A)$ is noetherian
and each $H^i(A)$ is finitely generated over $H^0(A)$; it is local
if $H^0(A)$ is local. Throughout, $A$ is a commutative nonpositive
noetherian local DG-ring with bounded cohomology and $H^0(A)\neq0$.
We write $R=H^0(A)$ and denote its maximal ideal by $\m$.

A DG $A$-module is a graded $A$-module $M$ with a differential
of degree $1$ satisfying $\partial^2=0$ and
\[
\partial(am)=\partial(a)m+(-1)^{|a|}a\partial(m).
\]
We regard $R$-modules as DG $A$-modules concentrated in degree
zero via the canonical map $A\to R$.

\subsection*{The derived category}
We write $\D(A)$ for the derived category of DG $A$-modules,
obtained by inverting quasi-isomorphisms, and use $\simeq$ for
isomorphisms in $\D(A)$. For $M\in\D(A)$, set
\[
\inf M=\inf\{i:H^i(M)\neq0\},\qquad
\sup M=\sup\{i:H^i(M)\neq0\},\qquad
\amp M=\sup M-\inf M.
\]
For an acyclic DG-module, these values are $\infty$, $-\infty$,
and $-\infty$, respectively. Our shift convention is
\[
H^i(M[n])=H^{i+n}(M).
\]

The full subcategory $\D^b(A)$ consists of the DG-modules with
bounded cohomology. We write $\D_f(A)$ for the full subcategory
of objects whose cohomology is finitely generated over $R$
in every degree, and set
\[
\D_f^b(A)=\D_f(A)\cap\D^b(A).
\]

\subsection*{Derived functors and $K$-flat DG-modules}
The derived tensor product and derived Hom are denoted by
$-\lotimes_A-$ and $\RHom_A(-,-)$. Thus
\[
\Ext_A^i(M,L)=H^i\bigl(\RHom_A(M,L)\bigr).
\]
A DG $A$-module $F$ is \emph{$K$-flat} if $X\otimes_A F$ is
acyclic for every acyclic DG $A$-module $X$. Every DG-module
admits a $K$-flat resolution. If $F\to M$ is such a resolution,
then
\[
M\lotimes_A L\simeq F\otimes_A L.
\]

\subsection*{Koszul DG-modules and localization}
Let $\mathbf x=x_1,\ldots,x_r$ be a sequence in $R$, and choose
lifts $\widetilde{\mathbf x}=\widetilde x_1,\ldots,\widetilde x_r$
in $A^0$. We write
\[
K(A;\mathbf x)
=
A\otimes_{A^0}K(A^0;\widetilde{\mathbf x}).
\]
Its isomorphism class in $\D(A)$ is independent of the chosen
lifts. For the empty sequence, we set $K(A;\varnothing)=A$.

The Koszul DG-module is
$K$-flat over $A$ by \cite[proof of Proposition~2.12]{ShaulKoszul}. More generally, if $B$ is a DG $A$-algebra
that is $K$-flat over $A$, then every finite Koszul DG-module
over $B$ is $K$-flat over $A$.

For $s\in A^0$, let
\[
A_s=A\otimes_{A^0}(A^0)_s
\]
be the DG localization. For every DG $A$-module $M$,
\[
H^i(M\otimes_A A_s)\cong H^i(M)_s.
\]
Consequently, $A_s$ is $K$-flat over $A$, and so is
$K(A_s;\mathbf x)$ for every finite sequence $\mathbf x$
in $H^0(A_s)$.

\subsection*{Projective and finitistic dimensions}
For $M\in\D(A)$, we use the projective dimension
\[
\pd_A M
=
\inf\left\{
q\in\mathbb Z
\ \middle|\
\begin{array}{l}
\Ext_A^i(M,L)=0\text{ for every }L\in\D^b(A)\\
\text{and every }i>q+\sup L
\end{array}
\right\};
\]
see \cite[Definition~2.1]{BSSW}. The infimum of the empty set
is $\infty$.

Following \cite[Definition~7.1]{BSSW}, the small and big finitistic
projective dimensions of $A$ are
\begin{align*}
\fpd(A)
&=
\sup\{\pd_A M+\inf M:
      0\not\simeq M\in\D_f^b(A),\ \pd_A M<\infty\},\\
\FPD(A)
&=
\sup\{\pd_A M+\inf M:
      0\not\simeq M\in\D^b(A),\ \pd_A M<\infty\}.
\end{align*}

\subsection*{Local cohomology and Cohen-Macaulayness}
Choose a finite sequence $\widetilde{\mathbf x}$ in $A^0$
whose image generates $\m$. For $t\in A^0$, let
\[
\check C(A^0;t)=[A^0\longrightarrow(A^0)_t]
\]
be the augmented \v{C}ech complex in degrees $0$ and $1$.
The augmented \v{C}ech complex on a sequence is the tensor
product of the corresponding one-element complexes.
Derived local cohomology is computed by
\[
\RG_{\m}(M)
\simeq
\check C(A^0;\widetilde{\mathbf x})\otimes_{A^0}M.
\]
This construction is independent of the chosen lifts and
generators; see \cite{ShaulCM}.

Following \cite{ShaulCM}, the DG-ring $A$ is
\emph{local-Cohen-Macaulay} if
\[
\amp\bigl(\RG_{\m}(A)\bigr)=\amp(A).
\]

\subsection*{Trivial extensions}

Let $B$ be a commutative DG-ring and let $M$ be a DG
$B$-module. The \emph{trivial extension} of $B$ by $M$,
denoted by $B\ltimes M$, has underlying graded module
\[
(B\ltimes M)^i=B^i\oplus M^i
\]
and differential
\[
d(b,m)=(d_B(b),d_M(m)).
\]
For homogeneous elements $(b,m)$ of degree $i$ and $(c,n)$
of degree $j$, multiplication is defined by
\[
(b,m)(c,n)
=
\bigl(bc,\;bn+(-1)^{ij}cm\bigr).
\]
With these operations and unit $(1,0)$, the trivial extension
is a commutative DG-ring. It is non-positive whenever both
$B$ and $M$ are concentrated in non-positive degrees.

\section{The construction}
In this section we generalize \cite[Example 7.2]{ShaulCM} and show that the resulting DG-rings are local-Cohen–Macaulay but have unequal finitistic projective dimensions.
\begin{theorem}\label{thm:main}
Let $k$ be a field and let $d,n\geq1$. If
\[
R=\frac{k[[x,y,z_1,\ldots,z_{d-1}]]}{(xy)},\qquad
N=R/(x),\qquad A=R\ltimes N[n],
\]
then $A$ is a commutative noetherian local-Cohen-Macaulay DG-ring with
\[
\dim H^0(A)=d,\qquad \amp(A)=n,\qquad
\fpd(A)=d-n,\qquad \FPD(A)=d.
\]
In particular, $\FPD(A)-\fpd(A)=n$.
\end{theorem}

\begin{proof}
By construction $A$ is a commutative, noetherian and local DG-ring with amplitude $n$.

The ring $R$ is a Cohen-Macaulay ring of dimension $d$. Moreover, the $R$-module $N\cong k[[y,z_1,\ldots,z_{d-1}]]$
is Cohen-Macaulay of dimension $d$.
By \cite[Theorem 7.3]{ShaulCM}, the DG-ring $A$ is local-Cohen-Macaulay, and \cite[Corollary 7.6]{BSSW} gives
$\fpd(A)=d-n$.

To compute the big finitistic dimension, consider the DG-module
\[
P=K(A_x;z_1,\ldots,z_{d-1}),
\]
the Koszul DG-module of the sequence $z_1,\ldots,z_{d-1}$ over $A_x$.
Since $xN=0$, it follows that
$A_x=R_x$, which is therefore concentrated in degree zero. Moreover,
\[
R_x\cong k[[x,z_1,\ldots,z_{d-1}]][x^{-1}].
\]
The sequence $z_1,\ldots,z_{d-1}$ is regular in $R_x$. Thus the Koszul complex resolves the quotient,
giving
\begin{equation}\label{eq:witness}
P\simeq R_x/(z_1,\ldots,z_{d-1})R_x\cong k((x)),
\end{equation}
where $k((x))\colonequals k[[x]][x^{-1}]$. In particular, $P$ is nonzero and has cohomology only in degree zero,
so $\inf P=0$.

We recall that
\cite[Lemma~6.4]{BSSW} states that if $\mathfrak p$ is a prime ideal
in a commutative noetherian ring $T$ with $\dim(T/\mathfrak p)>0$,
$s\notin\mathfrak p$ and $(\mathfrak p,s)\neq T$, and if
$a_1,\ldots,a_r\in\mathfrak p$ form a regular sequence in $T_s$,
then
\[
\pd_T\bigl(T_s/(a_1,\ldots,a_r)T_s\bigr)=r+1.
\]
We apply this with
\[
T=R,\qquad \mathfrak p=(y,z_1,\ldots,z_{d-1}),\qquad
s=x,\qquad (a_1,\ldots,a_r)=(z_1,\ldots,z_{d-1}).
\]
Therefore
\[
\pd_R k((x))=\pd_R\bigl(R_x/(z_1,\ldots,z_{d-1})R_x\bigr)=d.
\]
The DG localization $A_x$ is $K$-flat over $A$ and therefore so is $P$. Thus,
\[
R\lotimes_A P\simeq K(R_x;z_1,\ldots,z_{d-1})\simeq k((x)).
\]
Therefore, \cite[Corollary 2.3(i)]{BSSW} and the preceding calculation give
\[
\pd_A P=\pd_R k((x))=d.
\]
It follows that
\[
d=\pd_A P+\inf P\leq\FPD(A)\leq\dim R=d,
\]
where the first inequality is by definition, and the second by \cite[Theorem 7.2]{BSSW}.
Hence $\FPD(A)=d$, completing the proof.
\end{proof}

\bibliographystyle{amsplain}
\bibliography{biblio}
\end{document}